\documentclass[12pt]{amsart}
\usepackage[margin=1.5in]{geometry}
\usepackage{amsthm,amsmath,amssymb}
\theoremstyle{plain}
\newtheorem{theorem}{Theorem}[section]
\newtheorem{lemma}[theorem]{Lemma}
\newtheorem{corollary}[theorem]{Corollary}
\newtheorem{question}[theorem]{Question}
\theoremstyle{definition}
\newtheorem{remark}[theorem]{Remark}
\numberwithin{equation}{section}
\newcommand{\gen}[1]{\langle#1\rangle}
\newcommand{\biggen}[1]{\left\langle#1\right\rangle}
\newcommand{\abs}[1]{\lvert#1\rvert}
\newcommand*{\set}[1]{\{#1\}}
\newcommand*{\defset}[2]{\{\,#1\,\mid\,#2\,\}}
\newcommand{\QQ}{\mathbb Q}
\newcommand{\FF}{\mathbb F}
\newcommand{\Fqs}{\FF_q^\times}
\newcommand{\bx}{\bar{x}}

\newcommand{\bH}{\bar{H}}

\DeclareMathOperator{\PGL}{PGL}
\DeclareMathOperator{\PSL}{PSL}
\DeclareMathOperator{\SL}{SL}
\DeclareMathOperator{\SU}{SU}
\DeclareMathOperator{\PGammaL}{P\Gamma L}

\DeclareMathOperator{\trace}{tr}
\DeclareMathOperator{\adj}{adj}
\DeclareMathOperator{\aut}{Aut}
\usepackage{xcolor} 
\usepackage{url}
\usepackage[T1]{fontenc}
\usepackage{lmodern}
\usepackage[
  backend=biber,
  style=numeric,
  doi=true,
  url=false,
  giveninits=true,
  isbn=false
]{biblatex}
\usepackage[colorlinks,linkcolor=blue,citecolor=purple]{hyperref}
\title{Transitive sets of derangements in primitive actions of $\PSL_2(q)$}
\author{Peter M\"uller}
\address{Institute of Mathematics, University of W\"urzburg}
\email{peter.mueller@uni-wuerzburg.de}
\begin{document}
\begin{abstract}
  Problem 8.75 of the Kourovka Notebook \cite{kourovka_notebook},
  attributed to John G.~Thompson, asks the following: Suppose $G$ is a
  finite primitive permutation group on $\Omega$, and $\alpha$,
  $\beta$ are distinct points of $\Omega$. Does there exist an element
  $g\in G$ such that $\alpha^g=\beta$ and $g$ fixes no point of
  $\Omega$? A recent negative example is given in
  \cite{mueller___8.75}, where $G$ is the Steinberg triality group
  ${}^{3}D_{4}(2)$ acting primitively on $4{,}064{,}256$ points. At
  present this is the only negative example known. In this note we
  show that almost simple primitive permutation groups with socle
  isomorphic to $\PSL_2(q)$ do not give negative examples.
\end{abstract}
\maketitle
\section{Introduction}
It is an easy exercise using the orbit-counting theorem to show that
if $G$ is a transitive permutation group on a finite set $\Omega$,
then the set of derangements of $G$ generates a transitive subgroup.

It is also easy to see that in general the set of derangements,
together with the identity, need not be transitive on $\Omega$. A
small example is the alternating group $A_4$ acting on the $2$-subsets
of $\set{1,2,3,4}$: An element $g\in A_4$ moving $\set{1,2}$ to
$\set{3,4}$ moves every element in its natural degree $4$ action, so
it is a double transposition, and therefore fixes some $2$-subset.

While there are many more such examples, it seemed hard to find one
under the additional assumption that $G$ acts primitively. In fact the
following Problem 8.75 of the Kourovka Notebook
\cite{kourovka_notebook}, attributed to John G.~Thompson and
designated ``A known problem'', was open for more than 40 years.
\begin{question}\label{q:8.75}
  Suppose $G$ is a finite primitive permutation group on $\Omega$, and
  $\alpha$, $\beta$ are distinct points of $\Omega$. Does there exist
  an element $g\in G$ such that $\alpha^g=\beta$ and $g$ fixes no
  point of $\Omega$?
\end{question}
In \cite{mueller___8.75} we give a negative answer. The group there is
the Steinberg triality group ${}^{3}D_{4}(2)$ of order
$211{,}341{,}312 = 2^{12}\cdot 3^{4}\cdot 7^{2}\cdot 13$ in a
primitive permutation action on a set of size $4{,}064{,}256$.

Using Magma \cite{magma}, we checked that this is the only example if
we assume that $G$ is almost simple with socle of size less than or
equal $20{,}158{,}709{,}760$. It would have been computationally too
expensive to check all the groups with socle $\PSL_2(q)$ in this
range. The purpose of this note is to show that these groups do not
produce negative examples. In fact a somewhat stronger result holds.
\begin{theorem}\label{t:main}
  Let $q\ge4$ be a prime power and $G$ be almost simple with socle
  $S=\PSL_2(q)$, that is $S\le G\le\aut(S)$. Suppose that $G$ acts
  primitively on $\Omega$. Then for every distinct
  $\alpha, \beta\in\Omega$ there exists $s\in S$ such that
  $\alpha^s=\beta$ and $s$ fixes no point of $\Omega$.
\end{theorem}
If $G$ acts primitively on $\Omega$, then the possibilities for a point
stabilizer $S_\omega < S$ are known. While some cases are relatively
easy to handle, the more difficult ones arise when $S_\omega$ is
dihedral, or when $S_\omega$ is one of the polyhedral groups $A_4$,
$S_4$, or $A_5$. In the former case we require Hasse’s bound on the
number of points on elliptic curves over finite fields, and to rule out
the latter cases we derive trace identities for the binary polyhedral
groups $2A_4$, $2S_4$, and $2A_5$.

In some cases there are lengthy calculations which, in principle,
could be carried out by hand. Instead, we provide SageMath code at
\cite{dera_psl2_verify} that confirms the claims.
\section{Quartic hyperelliptic curves}
If $D$ is a plane algebraic curve over a field $F$ given by
$P(X, Y)\in F[X, Y]$, then $D(F)=\defset{(x, y)\in F^2}{P(x,y)=0}$
denotes the set of affine $F$-points on $D$.

Let $f(X)=a_0+a_1X+a_2X^2+a_3X^3+a_4X^4\in F[X]$ be a separable
polynomial of degree $4$ over a field $F$ of characteristic $\ne2$,
and $C$ be the plane algebraic curve defined by $Y^2=f(X)$. It is well
known that $C$ is $F$-birationally equivalent to an elliptic curve $E$
in Weierstraß normal form if and only if $C(F)$ is not empty or $a_4$
is a square in $F$.

As to explicit transformation formulae, there does not seem to be much
in the literature. The formulae in \cite[Chapter 10, Theorem
2]{Mordell:DE} and \cite{weil___remarques} do not cover the
characteristic $3$ case. Also, these sources do not discuss how $C(F)$
is related to $E(F)$. This is done in \cite[Section 3]{harada_lang},
but the treatment is somewhat complicated, it does not work in
characteristic $3$, and also fails in arbitrary characteristic (as
noted in \cite{connell___addendum}) if, in the notation from above,
$a_0$ is a square and $4a_0^2a_2=a_1^2$.

The purpose of this section is to provide a complete treatment of a
birational map $C\to E$, thereby keeping track of rational
points on both curves.

First, there are some easy transformations: If $(x,y)\in C(F)$, then a
shift in $X$ allows to assume $x=0$. So $a_0=\alpha^2$ and
$(\alpha, y)\in C(F)$. Upon replacing $X$ and $Y$ with $1/X$ and
$Y/X^2$ we may assume that $a_4=\alpha^2$ is a square. If $a_4=0$,
then we arrived at a Weierstraß form. Thus assume $a_4\ne0$. Upon
replacing $Y$ with $\alpha Y$, we may assume that $a_4=1$. Another
shift in $X$ allows to assume $a_3=0$. So we may assume that
$f(X)=a_0+a_1X+a_2X^2+X^4=(a_2/2+X^2)^2+a_1X+a_0-a_2^2/4$. Writing
$a=a_2^2/4-a_0$, $b=a_1$, and $c=a_2/2$, we see that we are reduced to
study curves of the form $Y^2=(X^2+c)^2 + bX - a$.

The transformation formulae below arise from a simplified version of
\cite{connell___addendum}, combined with an idea from \cite[Chapter
8(iii)]{cassels___ec}.

\begin{lemma}
  Let $f(X)=(X^2+c)^2 + bX - a\in F[X]$ be a separable polynomial over
  a field $F$ of characteristic $\ne2$. Then $g(U)=U^3 -4cU^2 + 4aU +
  b^2\in F[U]$ is separable.

  Let $C$ be the plane algebraic curve defined by $Y^2=f(X)$ and $E$ be
  the elliptic curve defined by $V^2=g(U)$.

  Consider the bivariate maps
  \begin{align*}
    \Phi:\;(X, Y) &\mapsto (U, V)=(2(X^2 + Y +c), 2XU + b)\\
    \Psi:\;(U, V) &\mapsto (X, Y)=\left(\frac{V - b}{2 U}, -X^{2}%
                    + \frac{1}{2}(U-2c)\right).
  \end{align*}
  Then the following holds:
  \begin{itemize}
  \item[(a)] If $b=0$, then $\Phi$ induces a bijective map
    $C(F)\to E(F)\setminus\set{(0, 0)}$ with inverse map induced by
    $\Psi$. (Note that $(0, 0)\in E(F)$ in this case.)
  \item[(b)] If $b\ne0$, then $(0,\pm b)\in E(F)$ and
    $(x_0, \pm y_0)\in C(F)$ for $x_0=\frac{a}{b}$,
    $y_0=c+x_0^2$. Furthermore, $\Phi$ induces a bijective map
    $C(F)\setminus\set{(x_0,-y_0)}\to E(F)\setminus\set{(0, -b), (0,
      b)}$ with inverse map induced by $\Psi$.
  \end{itemize}

  In particular, $\abs{C(F)}=\abs{E(F)}-1$ in both cases.
\end{lemma}
\begin{proof}
  The separability of $g(U)$ follows from that of $f(X)$, because both
  polynomials have the same discriminant. (Alternatively, this could
  also be derived without calculation of discriminants from the
  properties of $\Phi$ and $\Psi$.)
  
  Direct computations show that $\Phi(x,y)\in E(F)$ for all
  $(x,y)\in C(F)$ and $\Psi(u, v)\in C(F)$ for all
  $(u, v)\in E(F)\setminus\set{(0, -b), (0, b)}$.

  One verifies the identities $\Phi(\Psi(U, V))=(U, V)$ and
  $\Psi(\Phi(X, Y))=(X, Y)$ of bivariate rational functions; these
  identities hold without assuming the relations $V^2=g(U)$ or
  $Y^2=f(X)$.

  \begin{itemize}
    \item[(a)] Suppose that $b=0$. We need to show that there is no
      $(x, y)\in C(F)$ with $(u, v)=\Phi(x, y)=(0, 0)$. Suppose
      otherwise. From $y^2 - f(x) = 0$ and $u = 0$ we compute
  \begin{align*}
    0 &=%
        2(y^2-f(x)) + (x^2 - y + c)u\\
      &= 2(y^2 - x^4 - 2cx^2 - c^2 + a)%
        + (x^2 - y + c)(2x^2 + 2y + 2c)\\
    &= 2a. 
  \end{align*}
  Hence $a=b=0$, which contradicts the separability of $f(X)$.
\item[(b)] Now suppose that $b\ne0$. Assume that
  $(u, v)=\Phi(x,y )=(0, \pm b)$ for some $(x, y)\in C(F)$. From $u=0$
  we get $y=-(x^2+c)$. Plugging this into $y^2=f(x)$ gives
  $x=\frac{a}{b}=x_0$ and then $y=-(x^2+c)=-y_0$. The claim follows.
\end{itemize}
\end{proof}
\begin{lemma}\label{l:hasse}
  Let $q$ be an odd prime power and let $f\in\FF_q[X]$ be a separable
  polynomial of degree $\le4$. Then the number of pairs
  $(x,y)\in\FF_q^2$ with $y^2=f(x)$ is at most $q+1+2\sqrt{q}$.
\end{lemma}
\begin{proof}
  If $f$ has degree $\le2$, then an even better upper bound holds. If
  $\deg f=3$, then $Y^2=f(X)$ is an elliptic curve and Hasse's bound
  gives the slightly better upper bound $q+2\sqrt{q}$. (Note that we
  only count the affine points of the curve.)

  Now suppose that $\deg f=4$ and that $y^2=f(x)$ for some
  $x,y\in\FF_q$. Let $C'$ be the curve given by $Y^2=f(X)$. By the
  transformations described above, we transform $C'$ to the curve $C$
  as in the previous lemma. The only transformation on the way from
  $C'$ to $C$ which changes the number of affine $\FF_q$-points is the
  map $(X,Y)\mapsto(1/X,Y/X^2)$, which removes the (at most $2$)
  points with $X$-coordinate $0$. Thus
  $\abs{C'(\FF_q)}\le\abs{C(\FF_q)}+2$. Let $E$ be the elliptic curve
  from the previous lemma. So $\abs{C(\FF_q)}=\abs{E(\FF_q)}-1$. Thus
  $\abs{C'(\FF_q)}\le\abs{E(\FF_q)}+1$, and the claim follows from the
  Hasse bound for $\abs{E(\FF_q)}\le q+2\sqrt{q}$.
\end{proof}
\begin{remark}
  \begin{itemize}
  \item[(a)] One cannot improve the bound further. For instance, for
    $q=5$ there are $10>q+2\sqrt{q}$ pairs $(x,y)\in\FF_q^2$ with
    $y^2=3x^4+1$.
  \item[(b)] An alternative, more theoretical proof of the lemma could
    use the following: If $C$ is an arbitrary absolutely irreducible
    projective curve over $\FF_q$, and $\tilde C$ is a nonsingular
    curve which is equivalent to $C$ by a birational map over $\FF_q$,
    then $n(C)\le n(\tilde C)+r$, where $n$ counts the number of
    projective $\FF_q$-points, and $r$ is the number of singularities
    of $C$. See the proof of \cite[Proposition 2.3]{aubry_perret}.
  \end{itemize}
\end{remark}
\section{Some field theoretic lemmata}
\begin{lemma}
  Let $q$ be a prime power with $q\ge23$ if $q$ is odd. Suppose that
  for $a, d\in\FF_q$ the following holds: for every $r\in\Fqs$
  with $ar+d/r\ne0$ there is $s\in\Fqs$ with
  $ar+d/r=s+1/s$. Then either $a=d=0$ or $ad=1$.
\end{lemma}
\begin{proof}
First note that if exactly one of $a$ or $d$ is nonzero, then
$\Fqs\subseteq\defset{s+1/s}{s\in\Fqs}$. So the map
$\sigma: s\mapsto s+1/s$ is injective. But $\sigma(s)=\sigma(1/s)$, so
$s=1/s$ for all $s\in\Fqs$, and therefore $q\leq3$. The
case $q=3$ is indeed an exception, and for $q=2$ the claim is true
anyway.

Thus we assume that $a$ and $d$ are both nonzero.
  
We start with the easier case where $q$ is even. As $s+1/s=0$ for
$s=1$, we have $ar+d/r=s+1/s$ for some $s$ no matter whether $ar+d/r$
is $0$ or not. As the square map is an automorphism of $\FF_q$, there
is $r_0\in F$ with $d/a=r_0^2$.
  
From%
\[
  ar+d/r=ar+ar_0^2/r=ar_0(r/r_0+r_0/r)
\]
and $ad=(ar_0)^2$ we see that in order to prove the lemma, we may
assume $a=d$. 

Set $T=\defset{r+1/r}{r\in\Fqs}$. We see that multiplication by
$a$ permutes the elements of the set $T$. Thus $ax=x$ for
$x=\sum_{t\in T}t^2$. Let $\zeta\in\FF_q$ have order $q-1$. Then the
elements $\zeta^i+\zeta^{-i}$, $i=0,1,\dots,q/2-1$ are the elements
from $T$ without repetitions. Therefore
  \[
    x = \sum_{i=0}^{q/2-1}(\zeta^{2i}+\zeta^{-2i})%
    =
    \frac{\zeta^q-1}{\zeta^2-1}+\frac{\zeta^{-q}-1}{1/\zeta^2-1}%
    =\frac{\zeta-1}{\zeta^2-1}+\frac{1/\zeta-1}{1/\zeta^2-1}=1,
  \]
  hence $a=1$ and we are done.

  Now suppose that $q$ is odd. Note that $ar+d/r=0$ for at most two
  values of $r$. Thus for all but at most $2$ values of $r$, there is
  $s\in\Fqs$ with $ar+d/r=s+1/s$.

  Rewriting $s^2-(ar+d/r)s+1=0$ as $(2rs-ar^2-d)^2=(ar^2+d)^2-4r^2$
  shows that $(ar^2+d)^2-4r^2$ is a square for at least $q-3$ values
  of $r\in\Fqs$. Set $f(X)=(aX^2+d)^2-4X^2$. As $f$ has at most
  $4$ roots, we see that the curve $Y^2=f(X)$ has at least
  $2(q-3-4)+4=2q-10$ affine points over $\FF_q$. If $f$ is separable,
  then by Lemma \ref{l:hasse}, we get $2q-10\le q+1+2\sqrt{q}$, hence
  $q\le19$.

  Finally assume that $f(X)=(aX^2-2X+d)(aX^2+2X+d)$ is
  inseparable. The two factors are relatively prime (because the
  difference is $4X$). Furthermore, the discriminant of both factors
  is $4(1-ad)$, hence $ad=1$ and we are done.
\end{proof}
\begin{remark}
  One can check that $q=13$, $17$, and $19$ do not give
  exceptions. However, $q=3$, $5$, $7$, $9$, and $11$ are indeed
  exceptions. This makes it unlikely that there is a ``cleaner'' proof
  of the lemma for odd $q$.
\end{remark}
\begin{corollary}\label{c:split}
  Let $q\ge4$ be a prime power and $a, b, c, d\in\FF_q$ with
  $ad-bc=1$. Set
  $T=\defset{r+1/r}{r\in\Fqs}\cup\set{0}$. Suppose that
  $ar+d/r\in T$ and $-cr+b/r\in T$ for all $r\in\Fqs$. Then one
  of the following holds:
  \begin{itemize}
  \item[(a)] $a=d=0$ or $b=c=0$;
  \item[(b)] $q=9$ and $ar+d/r=2$ for some $r\in\Fqs$.
  \end{itemize}
\end{corollary}
\begin{proof}
  Using the previous lemma, one only need to check the cases where
  $q\le19$ is an odd prime power. This is most conveniently done with
  the SageMath code at \cite{dera_psl2_verify}.
\end{proof}
\begin{lemma}
  Let $q$ be a prime power with $q\ge23$ if $q$ is odd. For
  $F=\FF_{q^2}$ set $N=\defset{r\in F}{r^{q+1}=1}$. Suppose that for
  $0\ne a\in F$ the following holds: For all $r\in N$ with
  $ar+a^q/r\ne0$ there is $s\in N$ such that $ar+a^q/r = s+1/s$. Then
  $a\in N$.
\end{lemma}
\begin{proof} We first handle the easier case that $q$ is even.
As the square map is an automorphism of $F$, there is $r_0\in F$ with
$a^{q-1}=r_0^2$. Note that $a^{q-1}\in N$, hence $r_0\in N$.
  
From%
\[
  ar+a^q/r=ar+ar_0^2/r=ar_0(r/r_0+r_0/r)
\]
  we see that in order to prove the lemma, we may assume $a^q=a$.

  Set $T=\defset{r+1/r}{r\in N}$. We see that multiplication by $a$
  permutes the elements of the set $T$. Thus $ax=x$ for
  $x=\sum_{t\in T}t^2$. Let $\zeta\in F$ have order $q+1$. Then the
  elements $\zeta^i+\zeta^{-i}$, $i=0,1,\dots,q/2$ are the elements
  from $T$ without repetitions. Therefore
  \[
    x = \sum_{i=0}^{q/2}(\zeta^{2i}+\zeta^{-2i})%
    =
    \frac{\zeta^{2+q}-1}{\zeta^2-1}+\frac{\zeta^{-2-q}-1}{1/\zeta^2-1}%
    =\frac{\zeta-1}{\zeta^2-1}+\frac{1/\zeta-1}{1/\zeta^2-1}=1,
  \]
  hence $a=1$ and we are done.

  Now assume that $q$ is odd. An argument as above does not work
  because if $a^{(q-1)/2}$ is not in $N$, then
  $\defset{ar+a^q/r}{r\in N}$ is a proper subset of
  $\defset{s+1/s}{s\in N}$. (Note that $ar+a^q/r$ is fixed under
  $r\mapsto a^{q-1}/r$, so this map would not have a fixed point, in
  contrast to $s\mapsto 1/s$ which fixes $s+1/s$ for $s=\pm1$.)

  We use a finite field analog of the Cayley transform from complex
  numbers to parametrize the ``circle'' $N=\defset{r}{r\cdot r^q=1}$.
  Pick $0\ne\omega\in F$ with $\omega^q=-\omega$. Note that
  $\gamma=\omega^2$ is in $\FF_q$. Set $N'=N\setminus\set{1}$. One
  quickly checks that
  \[
    \FF_q\to N', z\mapsto\frac{z-\omega}{z+\omega}
  \]
  is bijective.

  Set $f(r)=ar+a^q/r$. We want to compute a lower bound for the size
  of $M=\defset{(r,s)\in N'\times N'}{f(r)=s+1/s}$. Note that given
  $c$, there are at most two values for $r$ with $f(r)=c$. Therefore,
  there are at least $\abs{N'}-6=q-6$ choices $r\in N'$ such that
  $f(r)\notin\set{-2, 0, 2}$. For these $r$, there is $s\in N'$ with
  $(r, s)\in M$. But then $(r,1/s)\ne(r,s)$ is also in $M$. Thus
  $\abs{M}\ge2(q-6)$.

  With $r=\frac{x-\omega}{x+\omega}$ and $s=\frac{y-\omega}{y+\omega}$
  we get that
  \[
    a\frac{X-\omega}{X+\omega} + a^q\frac{X+\omega}{X-\omega} =
    \frac{Y-\omega}{Y+\omega} + \frac{Y+\omega}{Y-\omega}
  \]
  has at least $2(q-6)$ solutions $(x,y)\in\FF_q\times\FF_q$. Write
  $a=u+v\omega$ with $u,v\in\FF_q$. Then $a^q=u-v\omega$. For later
  use we note that
  $a^{q+1}=a\cdot a^q=(u+v\omega)(u-v\omega)=u^2-v^2\gamma$.

  Clearing denominators and using $\omega^2=\gamma$, we arrive at the
  curve $P(X)Y^2 = \gamma Q(X)$, where
  \begin{align*}
    P(X) &= \left(u - 1\right) X^{2} - 2 v \gamma X + (u+1)\gamma\\
    Q(X) &= \left(u + 1\right) X^{2} - 2 v \gamma X + (u-1)\gamma.
  \end{align*}
  Note that $P, Q\in\FF_q[X]$.

  The discriminant of $P(X)$ and $Q(X)$ is
  $-4\gamma(u^2-v^2\gamma - 1)=-4\gamma(a^{q+1}-1)$ and
  $4\gamma(u^2-v^2\gamma - 1)=-4\gamma(a^{q+1}-1)$,
  respectively. Furthermore, the resultant of $P(X)$ and $Q(X)$ is
  $16\gamma^2(u^2-v^2\gamma)=16\gamma^2a^{q+1}$.

  We assume that $a^{q+1}\ne1$, for otherwise $a\in N$ and we are
  done. Then the product $P(X)Q(X)$ is a separable polynomial of
  degree $\le4$.

  Note that if $(x,y)$ is a point on $P(X)Y^2 = \gamma Q(X)$, then
  $P(x)\ne0$. Thus the points $(x,y)$ on the curve
  $P(X)Y^2 = \gamma Q(X)$ give different points $(x, yP(x))$ on the
  curve $Y^2=\gamma P(X)Q(X)$.

  We see that the curve $Y^2=\gamma P(X)Q(X)$ has at least $2(q-6)$
  affine points over $\FF_q$. On the other hand, by Lemma
  \ref{l:hasse}, this number is at most $q+1+2\sqrt{q}$. From this one
  quickly obtains $q\le19$.
\end{proof}
\begin{remark}
  One can check that $q=11$, $17$, and $19$ do not give exceptions.
  However, the cases $q=3$, $5$, $7$, $9$, $13$ are indeed
  exceptions. This is most easily seen for $q=3$, because then
  $\defset{s+1/s}{s\in N}=\FF_q$, so the condition on $a$ is vacuous.
\end{remark}
\begin{corollary}\label{c:nonsplit}
  Let $q$ be a prime power. Set $N=\defset{r\in\FF_{q^2}}{r^{q+1}=1}$
  and $T=\defset{r+1/r}{r\in N}$. For $a, b\in\FF_{q^2}$ with
  $a^{q+1} + b^{q+1} = 1$ suppose that $ar+a^q/r\in T$ and
  $br+b^q/r\in T$ for all $r\in\Fqs$. Then $a=0$ or $b=0$.
\end{corollary}
\begin{proof}
  If $q\ge23$, then the previous lemma shows that $a^{q+1}=1$, hence
  $b=0$, or $b^{q+1}=1$, hence $a=0$. The cases $q\le19$ are
  checked in \cite{dera_psl2_verify}.
\end{proof}
\section{The groups of quaternions $2A_4$, $2S_4$, and $2A_5$}
\begin{lemma}
  Let $F$ be a field and $\bH$ be a subgroup of $\PSL_2(F)$ isomorphic
  to $A_4$, $S_4$, or $A_5$. Let $H$ be the preimage of $\bH$ in
  $\SL_2(F)$ and $x, y\in F^{2\times 2}$. If the characteristic of $F$
  is $0$ or relatively prime to $\abs{\bH}$, then
  \[
    \sum_{h\in H}\trace(hy)\cdot\trace(hx)^{2m+1}%
    =\abs{\bH}\cdot c_m\cdot%
    (\trace(x)\cdot\trace(y)-\trace(x\cdot y))\cdot\det(x)^m,
  \]
  where $m\in\set{0, 1}$ if $\bH\cong A_4$, $m\in\set{0, 1, 2}$ if
  $\bH\cong S_4$, $m\in\set{0, 1, 2, 3, 4}$ if $\bH\cong A_5$, and
  $c_0, c_1, c_2, c_3, c_4$ = $1, 2, 5, 14, 42$ are the first five
  Catalan numbers (with index shifted by $1$).
\end{lemma}
\begin{proof}
  We assume that the characteristic of $F$ is $0$ or relatively prime
  to $\abs{\bH}$. Let $\bar F$ be an algebraic closure of $F$. It is
  well known that $\PSL_2(\bar F)$ contains exactly one conjugacy
  class of subgroups isomorphic to $\bH$, see e.g.~\cite[Proposition
  4.1]{beauville___pgl2}. (In fact more is true. By \cite[Proposition
  4.2]{beauville___pgl2} $\PGL_2(F)$ contains at most one conjugacy
  class of subgroups isomorphic to $\bH$.)

  If we replace $H$ by $H^g$ for some $g\in\SL_2(\bar F)$, then
  in view of $\trace(g^{-1}hgx)=\trace(hgxg^{-1})$ and
  $\trace(g^{-1}hgy)=\trace(hgyg^{-1})$, this amounts to replacing $x$
  and $y$ with $gxg^{-1}$ and $gyg^{-1}$, which however does not
  change the right hand side.

  Thus the claimed identity is independent of the chosen isomorphic
  copy of $H$ in $\SL_2(\bar F)$.

  In $\bar F$ we pick elements $i, \rho$, and $\sigma$ with $i^2=-1$,
  $\rho^2=2$, and $\sigma^2+\sigma=1$. Let $E\in\bar F^{2\times2}$ be
  the identity matrix, and set
  \[ I = \begin{pmatrix}i & 0\\0 & -i\end{pmatrix},\;\;%
    J = \begin{pmatrix}0 & 1\\-1 & 0\end{pmatrix},\;\;%
    K = IJ.
\]
Then $I^2=J^2=K^2=-E$ and $JI=-IJ$. Note that
$\set{\pm E, \pm I, \pm J, \pm K}$ is the quaternion group
$Q_8$. Following \cite{conway_smith}, we set
\[
  W = \frac{-E+I+J+K}{2}
\]
and
\begin{align*}
  2A_4 &= \gen{I, W}\\
  2S_4 &= \biggen{\frac{J + K}{\rho}, W}\\
         2A_5 &= \biggen{\frac{I + \sigma J + (\sigma+1)K)}{2}, W}.
\end{align*}
Then the images of $2A_4$, $2S_4$, and $2A_5$ in $\PSL_2(\bar F)$ are
isomorphic to $A_4$, $S_4$, and $A_5$, respectively.

With this explicit description of the groups
$H\in\set{2A_4, 2S_4, 2A_5}$, we verify the computation with the
SageMath script at \cite{dera_psl2_verify}. Note that the computations
are in characteristic $0$. But the only denominators which appear in
the calculations are $2$ (or $\rho$ where $\rho^2=2$), so the
identities which \cite{dera_psl2_verify} verifies hold true modulo the
characteristic of $F$. More precisely, \cite{dera_psl2_verify} does
the following: Let $F$ be the number field $\QQ(i)$, $\QQ(i, \rho)$,
or $\QQ(i, \sigma)$ according to $\bH=A_4$, $\bH=S_4$, or
$\bH=A_5$. Even though the given matrices generate finite groups,
SageMath (and its backend Gap) cannot compute the generated groups
over such number fields. The helper function \texttt{gener(L)}
implements a naive straightforward algorithm to compute the group
generated by the list \texttt{L} of matrices from $F^{2\times 2}$. The
remaining code which verifies the stated identities should be self
explanatory.
\end{proof}
\begin{remark}
  \begin{itemize}
    \item[(a)] The stated identities do not generalize to bigger
      values of $m$.
    \item[(b)] With a little more effort one can show that the lemma
      still holds under the weaker assumption that the characteristic
      of $F$ is not $2$. For the cases $A_4$ and $S_4$ essentially the
      same argument goes through, because $\PSL_2(\bar F)$ has only
      one conjugacy class of subgroups isomorphic to
      $V=C_2\times C_2$, the normalizer of $V$ in $\PSL_2(\bar F)$ is
      $S_4$, and $A_4$ is the derived subgroup of $S_4$. The case
      $A_5$ requires an additional argument if $F$ has characteristic
      $5$, since, besides the conjugacy class of groups considered
      above, there is another conjugacy class represented by
      $\PSL_2(5)$.
    \item[(c)] There is probably a more conceptual explanation of the
      identity. Let the entries of $x$ and $y$ be algebraically
      independent over $F$ and $f(x, y)$ be the left hand side, which
      is a polynomial of degree $2m+2$ in $8$ variables.
      
      Note that for $h\in H$ (recall that $H\le\SL_2(\bar F)$), we
      have $f(h'x, h'y)=f(x, y)$, so $f$ is an invariant of
      $H$. Working out the low degree invariants of $H$ could
      provide a more conceptual explanation. (At a first glance, it
      might not be obvious that the right hand side is an invariant
      under simultaneous left multiplication of $x$ and $y$ by
      $h$. But this follows from the identity
      $\trace(x)\cdot\trace(y)-\trace(x\cdot y)=\trace(\adj(y)\cdot
      x)$, where $\adj(y)$ denotes the adjugate of $y$.)
    \end{itemize}
\end{remark}
\begin{lemma}\label{l:polyhedral}
  Let $F$ be a field and $\bH$ be a subgroup of $\PSL_2(F)$ isomorphic
  to $A_4$, $S_4$, or $A_5$. Assume that the characteristic of $F$ is
  $0$ or relatively prime to $\abs{\bH}$. Let $H$ be the preimage of
  $\bH$ in $\SL_2(F)$. Set
  $T=\defset{\trace(h)}{h\in H}\setminus\set{-2, 2}$, and for
  $x = \begin{pmatrix}a & b\\c & d\end{pmatrix}\in\SL_2(F)$ set
  $y = \begin{pmatrix}0 & 0\\c & d\end{pmatrix}$. Then
  \[
    \sum_{h\in H}\left(\trace(hy)\prod_{t\in T}(\trace(hx)-t)\right)%
    =\abs{\bH}.
  \]
\end{lemma}
\begin{proof}
  With this choice of $x$ and $y$ we have
  \[\trace(x)\cdot\trace(y)-\trace(x\cdot y)=ad-bc=\det(x)=1,\]
  so the previous lemma yields
     \[
       \sum_{h\in H}\trace(hy)\cdot\trace(hx)^{2m+1}%
       =\abs{\bH}\cdot c_m.
     \]
     Let $Z$ be a variable over $F$. Then \cite{dera_psl2_verify} verifies
     \[
       \prod_{t\in T}(Z-t) =
       \begin{cases}
         Z^3 - Z,&\text{ if }\bH\cong A_4\\
         Z^5 - 3Z^3 + 2Z,&\text{ if }\bH\cong S_4\\
         Z^7 - 4Z^5 + 4Z^3 - Z,&\text{ if }\bH\cong A_5.\\
       \end{cases}
     \]
This yields
     \[
       \sum_{h\in H}\left(\trace(hy)\prod_{t\in
           T}(\trace(hx)-t)\right)%
       =
       \begin{cases}
       \abs{\bH}(c_1-c_0)=\abs{\bH},&\text{ if }\bH\cong A_4\\
         \abs{\bH}(c_2-3c_1+2c_0)=\abs{\bH},&\text{ if }\bH\cong S_4\\
         \abs{\bH}(c_3-4c_2+4c_1-c_0)=\abs{\bH},&\text{ if }\bH\cong A_5.\\
       \end{cases}
     \]
\end{proof}
\section{Proof of Theorem \ref{t:main}}
The maximal subgroups of almost simple groups with socle $\PSL_2(q)$
are known; see \cite[Table~8.1, Table~8.2]{bray_etal___maximal} or
\cite{giudici___psl2}. The results in \cite{bray_etal___maximal} and
\cite{giudici___psl2} are more detailed than those in the following
lemma, which lists only the restrictions on $q$ that we shall use.
\begin{lemma}\label{l:giudici}
  Let $q\ge4$ be a prime power and $S=\PSL_2(q)\le
  G\le\PGammaL_2(q)$. Let $M$ be a maximal subgroup of $G$ with
  $S\not\le M$. Set $\bH=M\cap S$. Then one of the following holds.
  \begin{itemize}
  \item[(a)] $[S:\bH] = q+1$,
  \item[(b)] $[S:\bH] = q(q+1)/2$,
  \item[(c)] $[S:\bH] = q(q-1)/2$,
  \item[(d)] $\bH\cong A_4$, and $q$ is and odd prime,
  \item[(e)] $\bH\cong S_4$, and $q$ is a prime with $q\equiv\pm1\mod{8}$,
  \item[(f)] $\bH\cong A_5$, and $q\equiv\pm1\mod{10}$ and $q=9$ if $q$ is
    power of $3$,
  \item[(g)] $\bH\cong\PSL_2(r)$ or $\bH\cong\PGL_2(r)$, where $q=r^e$
    for $e\ge2$.
  \end{itemize}
  Furthermore, in each case $\bH$ is unique up to conjugacy in
  $\aut(S)$ (with two classes in case (g) if $r$ is odd).
\end{lemma}
Let $S$ and $G$ be as in Theorem \ref{t:main}, $M$ be a stabilizer of
a point in $\Omega$, and $\bH=M\cap S$. Then $M$ is a maximal subgroup
of $G$ by primitivity. We handle the cases (a) to (g) separately.

Case (a) is very easy: Here the action of $S$ on $\Omega$ is the
natural action, which is doubly transitive. Pick distinct
$\alpha, \beta\in\Omega$ and an arbitrary derangement $y$ in $S$. By
double transitivity of $S$, there is $s\in S$ with $\alpha^s=\alpha$
and $\alpha^{ys}=\beta$. (Note that $\alpha^y\ne\alpha$.) Then $g=y^s$
is a derangement with $\alpha^g=\alpha^{s^{-1}ys}=\beta$.

The remaining cases (b) to (g) require some work. Suppose that there
is a counterexample. Then there exist distinct
$\alpha, \beta\in\Omega$ such that no derangement in $S$ moves
$\alpha$ to $\beta$. Let $\bH$ be the stabilizer in $S$ of $\alpha$
and $\bx\in S$ with $\alpha^{\bx}=\beta$. Then $\bH\bx$ is the set of
elements which moves $\alpha$ to $\beta$. As this set does not contain
a derangement, every element in the coset $\bH\bx$ is conjugate to an
element from $\bH$.

This property is preserved upon applying an automorphism of $S$ to
$\bH$ and $\bx$. Therefore, in each of the cases listed in the
lemma, we may make a specific choice of $\bH$.

Let $H$ be the preimage of $\bH$ in $\SL(2, q)$, and
$x\in\SL(2, q)$ be a preimage of $\bx$. Thus $hx$ is conjugate to an
element from $H$ for every $h\in H$. This implies
\begin{equation}
  \label{eq:trace1}
  \defset{\trace(hx)}{h\in H}\subseteq%
  \defset{\trace(h)}{h\in H}.  
\end{equation}
We will also need a refinement of this set inclusion. Suppose that
$\trace{hx}=\pm2$ for some $h\in H$. Then the characteristic
polynomial of $hx$ is $X^2\pm2X+1=(X\pm1)^2$, so the order of
$hx$ is divides the characteristic of $\FF_q$. Thus if the
characteristic of $F$ does not divide $\abs{H}$, then
$hx\in H$, hence $x\in H$, a contradiction. So we get
\begin{equation}
  \label{eq:trace2}
  \defset{\trace(hx)}{h\in H}\subseteq%
  \defset{\trace(h)}{h\in H}\setminus\set{-2, 2}\text{ if }%
  \gcd(q, \abs{H})=1.
\end{equation}
\subsection{Case (b)}
In this case $H=\defset{h_r}{r\in\Fqs}\cup\defset{h_r'}{r\in\Fqs}$
with
\[
  h_r=\begin{pmatrix}r & 0\\0 & 1/r\end{pmatrix},\;\;
  h_r'=\begin{pmatrix}0 & -r\\1/r & 0\end{pmatrix}.\] Thus
$\defset{\trace(h)}{h\in H}=\defset{r+1/r}{r\in\Fqs}\cup\set{0}$. Write
$x=\begin{pmatrix}a & b\\c & d\end{pmatrix}$ with $ad-bc=1$. Then
$\trace(h_rx)=ar+d/r$ and $\trace(h_r'x)=-cr+b/r$.

By \eqref{eq:trace1} and Corollary \ref{c:split} we have either
$a=d=0$ or $b=c=0$, but then $x\in H$, a contradiction, or $q=9$ and
$2=ar+d/r=\trace(h_rx)$. But $\gcd(q, \abs{H})=\gcd(9, 8)=1$, which
contradicts \eqref{eq:trace2}.
\subsection{Case (c)} 
To rule out this case, it is convenient to use the isomorphism
\[
  \SL_2(\FF_q)\cong\SU_2(\FF_{q^2})%
  =\defset{%
    \begin{pmatrix}a & b\\-b^q & a^q\end{pmatrix}}{a, b\in\FF_{q^2},
    a^{q+1}+b^{q+1}=1}.
  \]
  Set $N=\defset{r\in\FF_{q^2}}{r^{q+1}=1}$. Note that $r^q=1/r$ for
  $r\in N$. Working in $\SU_2(\FF_{q^2})$, we may assume that
  $H=\defset{h_r}{r\in N}\cup\defset{h_r'}{r\in N}$ with
\[
  h_r=\begin{pmatrix}r & 0\\0 & 1/r\end{pmatrix},\;\;
  h_r'=\begin{pmatrix}0 & -1/r\\r & 0\end{pmatrix}.\]

Thus
$T:=\defset{\trace(h)}{h\in H}=\defset{r+1/r}{r\in
  N}\cup\set{0}$. Write
$x=\begin{pmatrix}a & b\\-b^q & a^q\end{pmatrix}$ with
$a^{q+1}+b^{q+1}=1$. Then $\trace(h_rx)=ar+a^q/r$ and
$\trace(h_r'x)=br+b^q/r$.

Now \eqref{eq:trace1} and Corollary \ref{c:nonsplit} show $a=0$ or
$b=0$, hence $x\in H$.
\subsection{Case (d)}
As $q$ is an odd prime and $q\ge4$ we have
$\gcd(q,\abs{H})=\gcd(q,24)=1$. Thus \eqref{eq:trace2} applies
here. For $T=\defset{\trace(h)}{h\in H}\setminus\set{-2, 2}$ we get
$\prod_{t\in T}(\trace(hx)-t)=0$ for all $h\in H$. Lemma
\ref{l:polyhedral} then yields the contradiction $0=12$.
\subsection{Case (e)}
As $q$ is a prime with $q\equiv\pm1\mod{8}$ and $q\ge4$ we get
$\gcd(q,\abs{H})=\gcd(q,48)=1$. Thus \eqref{eq:trace2} applies and
gives a contradiction as in case (d) using Lemma \ref{l:polyhedral}.
\subsection{Case (f)}
If $q$ is a power of $3$, then $q=9$. In this case $S\cong A_6$ and
$\bH\cong A_5$. So $S$ acts as $A_6$ in its natural action on $6$
points. As this action is doubly transitive, the argument from case
(a) yields a contradiction.

In the remaining cases we have $\gcd(q,\abs{H})=\gcd(q,120)=1$. Thus
\eqref{eq:trace2} applies again and gives a contradiction as in case (d)
using Lemma \ref{l:polyhedral}.
\subsection{Case (g)}
In this case we may assume $\SL_2(\FF_r)\subseteq H$, so
\eqref{eq:trace1} implies\[ \trace(hx)\in\FF_r%
  \text{ for all }h\in\SL_2(\FF_r).
\]
But then $\trace(yx)\in\FF_r$ for every $\FF_r$-linear combination $y$
of matrices from $\SL_2(\FF_r)$. Note that %
$\begin{pmatrix}1 & 0\\0 & 0\end{pmatrix}$, %
$\begin{pmatrix}0 & 1\\0 & 0\end{pmatrix}$, %
$\begin{pmatrix}0 & 0\\1 & 0\end{pmatrix}$, %
$\begin{pmatrix}0 & 0\\0 & 1\end{pmatrix}$ %
are such linear combinations. This forces $x\in\SL_2(\FF_r)$, so
$x\in H$, a contradiction.
\printbibliography
\end{document}